\documentclass[a4paper, 11pt]{amsart}

\usepackage[utf8]{inputenc}
\usepackage{lmodern}
\usepackage{mathrsfs}
\usepackage[english]{babel}
\usepackage{graphicx}
\usepackage{color}

\usepackage{enumerate} 
\usepackage{mdwlist} 
\usepackage{microtype}

\usepackage[all]{xy}
\usepackage{tikz-cd}
\usepackage{verbatim}
\usepackage{amsmath,amsfonts,amssymb,amsthm} 
\usepackage{bm} 
\usepackage{hyperref} 
\usepackage{mathtools}

\theoremstyle{plain}
\newtheorem{theorem}{Theorem}[section]
\newtheorem{lemma}[theorem]{Lemma}
\newtheorem{proposition}[theorem]{Proposition}
\newtheorem{corollary}[theorem]{Corollary}
\theoremstyle{definition}
\newtheorem{definition}[theorem]{Definition}
\newtheorem{remark}[theorem]{Remark}
\newtheorem{example}[theorem]{Example}

\newcommand{\ZZ}{\mathbb{Z}} 
\newcommand{\QQ}{\mathbb{Q}} 
\newcommand{\RR}{\mathbb{R}} 
\newcommand{\VV}{\mathbb{V}} 
\newcommand{\PP}{\mathbb{P}} 
\renewcommand{\AA}{\mathbb{A}} 
\newcommand{\sheaf}[1]{\mathscr{#1}} 
\newcommand{\OO}{\sheaf{O}}  
\newcommand{\Gm}{\mathbb{G}_m} 

\newcommand{\kk}{\Bbbk}

\newcommand{\cL}{\mathcal{L}}

\newcommand{\cX}{\mathcal{X}}
\newcommand{\cZ}{\mathcal{Z}}
\newcommand{\Sym}{\mathrm{Sym}}
\newcommand{\SymXC}{\mathrm{Sym}^n(X[n]/C[n])}
\newcommand{\HilbXC}{\mathrm{Hilb}^n(X[n]/C[n])}

\newcommand{\lra}{\longrightarrow}
\newcommand{\wt}{\mathrm{wt}}

\DeclareMathOperator{\conv}{Conv}

\DeclareMathOperator{\Pic}{Pic}

\DeclareMathOperator{\Spec}{Spec}

\DeclareMathOperator{\Proj}{Proj}
\DeclareMathOperator{\Hilb}{Hilb}

\newcommand{\into}{\hookrightarrow}

\renewcommand{\geq}{\geqslant}
\renewcommand{\leq}{\leqslant}

\title[Relative VGIT]{Relative VGIT and an application to degenerations of Hilbert schemes}
\date{\today}

\begin{document}

\author[L.~H.~Halle]{Lars H. Halle}
\address{University of Copenhagen\\ 
Department of Mathematical Sciences\\
Universitetsparken 5\\ 
2100 Copenhagen\\ 
Denmark}
\curraddr{University of Bologna\\
Department of Mathematics\\
Piazza Porta S.~Donato 5\\
40126 Bologna\\
Italy}
\email{larshalvard.halle@unibo.it}

\author[K.~Hulek]{Klaus Hulek}
\address{Leibniz Universit\"at Hannover\\
Institut f\"ur algebraische Geometrie\\
Welfengarten 1\\
30167 Hannover\\
Germany}
\email{hulek@math.uni-hannover.de}

\author[Z.~Zhang]{Ziyu Zhang}
\address{ShanghaiTech University\\
Institute of Mathematical Sciences\\
393 Middle Huaxia Road\\
201210 Shanghai\\
P.R.China}
\email{zhangziyu@shanghaitech.edu.cn}



\thanks{Klaus Hulek  is grateful to DFG for partial support under grant Hu 337/7-1.}

\begin{abstract}
	We generalize the classical semi-continuity theorem for GIT (semi)stable loci under variations of linearizations to a relative situation of an equivariant projective morphism $X \to S$ over an affine base $S$. As an application to moduli problems, we consider degenerations of Hilbert schemes, and give a conceptual interpretation of the (semi)stable loci of the degeneration families constructed in \cite{GHH-2016}.
\end{abstract}

\maketitle

\section{Introduction}\label{sec:intro}

\subsection{Background}

Geometric invariant theory (GIT) is a very versatile method for constructing moduli spaces. Classical GIT concerns the action of a reductive group $G$ on a projective variety $X$ along with an 
ample line bundle $L$ on $X$. The action of $G$ on $L$ determines a semistable locus $X^{ss}$ on $X$, with a quotient $X^{ss}/G$ defined as the GIT quotient $X /\!/ G$ with respect to $L$. 
The main tool for determining the (semi)stability of a point in $X$ is the Hilbert-Mumford criterion, which reduces the problem to a purely numerical computation of the weights of the actions of 
some $1$-parameter subgroups of $G$ on the fibres of $L$ at fixed points. We recommend \cite{GIT} for a detailed account of the classical theory.

The GIT quotient $X/\!/G$ depends on the line bundle $L$ and the chosen $G$-linearization. This was explored in detail in the late 1990's by Dolgachev and Hu \cite{DH98} as well as by Thaddeus \cite{Tha-1996} extending
GIT to VGIT. Their results were further sharpened by Ressayre \cite{Res00}. These techniques have become essential tools in the study of the birational geometry of moduli spaces. 
One important aspect of the theory is the semi-continuity property of the GIT (semi)stable loci under a variation of the $G$-linearized ample line bundle $L$. Roughly speaking, let $L_t$ be 
a family of $G$-linearized ample line bundles on $X$, then for any $t$ sufficiently close to $0$, we have $$X^s(L_0) \subseteq X^s(L_t) \subseteq X^{ss}(L_t) \subseteq X^{ss}(L_0);$$ see \cite[Proposition 4]{Res00} 
and also \cite[Lemma 3.10]{Laza-2013}, or \cite[Theorem 1.6.2.2]{Sch08}. 

We point out that the classical GIT was sometimes also formulated for schemes $X/S$, where $S = \Spec (R)$ is an affine variety, together with an action of a reductive group $G/S$, for example in \cite{Sesh72} and \cite{Tha-1996}. This setup is different from the equivariant situation, which we will be discussing now.

Here we are concerned with a generalization of the classical GIT in the equivariant setting. Namely, we consider a $G$-equivariant projective morphism $X \to S$ over an affine base $S$, together with a $G$-linearized relative ample line bundle $L$ on $X$. 
In other words, the group $G$ is now also allowed to act non-trivially on $S$. Then the classical GIT for a projective variety corresponds to the case where $S=\Spec(k)$. 
In \cite{GHH-2015}, M.~G.~Gulbrandsen and the first two named authors generalized GIT techniques, in particular the Hilbert-Mumford numerical criterion, to this equivariant situation, which was then applied to construct a degeneration family of Hilbert schemes of points on surfaces in \cite{GHH-2016}.

We would like to mention that the case of an equivariant $G$-morphism $f: X \to Y$ has also been considered in the literature by various authors in  different settings: 
Reichstein \cite{Reich} treats equivariant morphisms $f\colon X\to Y$
between projective varieties, and Hu \cite{HuRel96} uses symplectic techniques to
study equivariant morphisms between quasi-projective varieties. 
The question studied there concerns the functoriality of the (semi)stable loci.
Both the semi-continuity and the functoriality of (semi)stable loci will be further developed in the present paper.

\subsection{Main results}

In the present paper, we will further generalize the above equivariant setting to allow $G$-linearized line bundles $L$ which are not necessarily ample. Indeed, we will study two different generalizations of the notion of (semi)stability with respect to $L$, both of which coincide with the classical counterpart when $L$ is ample.

The first generalization is called \emph{weak (semi)stability}; see Definition \ref{def:new}. We will show that the loci of weakly (semi)stable points satisfy nice functorial properties along two classes of projective morphisms; see Lemmas \ref{lem:quot} and \ref{lem:fun} for the precise statements. 

The second generalization is called \emph{numerical (semi)stability}, which was introduced previously in \cite{CJR14}; see Definition \ref{def:numericalsemistable}. We will show that the classical semi-continuity result in VGIT can be generalized to allow polarizations that are not necessarily ample, if we replace the classical (semi)stability by numerical (semi)stability. More precisely, we will prove

\begin{theorem}[Theorem \ref{thm:maincont}]
	\label{thm:mainintro}
	Let $f \colon X \to S$ be a projective morphism over an affine variety $S$, and $G$ a reductive 
	group acting equivariantly on $X$ and $S$. Let $L_0$ and $L_1$ be $G$-linearized line bundles on $X$ and $L_t = L_0^{1-t} \otimes L_1^t$ for every $0 \leqslant t \leqslant 1$. 
	Assume that $L_t$ is ample for $0 < t \leqslant 1$. Then we have
	$$ X^{ns}(L_0) \subseteq X^s(L_t) \subseteq X^{ss}(L_t) \subseteq X^{nss}(L_0) $$
	for $0 < t \ll 1$.
\end{theorem}

The proof of the theorem mainly follows the idea in \cite{DH98,Res00}. It requires two main ingredients, which were both established only for a single projective variety $X$ in \cite{DH98}, hence require some extra work in the relative setting. The first ingredient is the continuity of the $M$-function, as defined in \eqref{eqn:M-func}, with respect to the choice of the line bundle; see Proposition \ref{prop:M-cont}. The second ingredient states that for a fixed $G$-action on $X$ with respect to different $G$-linearized ample line bundles, there are only finitely many possible semistable loci; see Proposition \ref{prop:fin-ss2}. Both results may be of independent interest, and we combine them with the relative Hilbert-Mumford criterion proved in \cite{GHH-2015} to conclude Theorem \ref{thm:mainintro}.

\subsection{Applications}

We would like to mention an application of the above results, which in fact was our original motivation for studying the functoriality and semi-continuity of (semi)stable loci with respect to both weaker versions of (semi)stability in the relative setting.

Our aim was to construct good degenerations of Hilbert schemes of points, which was carried out 
 in \cite{GHH-2016}. More precisely, let $$f \colon X \longrightarrow C$$ be a projective strict simple degeneration over an affine curve $C$. By Jun Li's technique of expanded degenerations (see e.g. \cite{li-2013}), 
 we can obtain an expanded family $$f[n] \colon X[n] \longrightarrow C[n]$$ with an equivariant action of $G[n] = (\Gm)^n$. The relative Hilbert scheme of points $$\Hilb^n(X[n]/C[n]) \longrightarrow C[n]$$ is also 
 equipped with a $G[n]$-action. By carefully choosing a suitable $G[n]$-linearized ample line bundle, we can obtain a GIT quotient $$I^n_{X/C} = \Hilb^n(X[n]/C[n])/\!/G[n]$$ which completes 
 the family $\Hilb^n(X/C) \to C$ over the open subset of $C$ where $f$ is smooth. We proved that this is a degeneration of Hilbert schemes with good geometric properties, 
 see e.g. \cite[Theorem 5.9, Corollary 5.16]{GHHZ-2018}.
A crucial point of this approach was the detailed analysis of the (semi)stable locus $\Hilb^n(X[n]/C[n])^{ss}$. Intriguingly, we discovered that the property whether a point $[Z] \in \Hilb^n(X[n]/C[n])$ is (semi)stable or not 
only depends on its underlying $0$-cycle; see \cite[Theorem 2.9]{GHH-2016}. In the present paper, we will combine Theorem \ref{thm:mainintro} and Lemma \ref{lem:fun} to give 
a conceptual explanation of this fact; see Corollary \ref{cor:pullbackagree} for details.

\subsection{Structure of the paper}

The paper is organized as follows. In Section \ref{sec:generalresults} we give our definition of the weak (semi)stability of a point on a variety $X$ under the action of a reductive group $G$ via a linearization $L$, 
and prove functoriality of the (semi)stable loci along certain types of projective morphisms. In Section \ref{sec:relVGIT}, we first recall the definition of the numerical (semi)stability, then 
establish the two main ingredients required for Theorem \ref{thm:mainintro}. Their proofs in the case of torus actions are given in \S \ref{subsec:contM} and \S \ref{subsec:finite-ss} respectively, and in 
the case of general reductive group actions in \S \ref{subsec:reductive}. We prove Theorem \ref{thm:mainintro} and two other variants of it in \S \ref{subsec:mainproof}. Finally, Section \ref{sec:application} 
is devoted to the main application of Theorem \ref{thm:mainintro}. After recalling the necessary notations required in the setting of expanded degenerations in \S \ref{subsec:appsetup}, we 
explain in \S \ref{subsec:ExpStep1} and \S \ref{subsec:ExpStep2} in two steps why the (semi)stability for a point in $\Hilb^n(X[n]/C[n])$ only depends on the underlying cycle.

After the first version of our paper was published on arXiv, Alexander H. W. Schmitt \cite{Schpriv} informed us that he has since obtained an independent proof of Theorem \ref{thm:mainintro}.

\subsection*{Acknowledgements}
The authors would like to thank Alexander H. W. Schmitt for pointing out the reference \cite{Sch08}.	Klaus Hulek would like to thank S\'andor Kov\'acs, Thomas Peternell and Charles Vial for 
discussions on higher dimensional birational geometry and algebraic cycles. We are also grateful to the anonymous referee for several suggestions for improving the presentation of the paper, and 
in particular, for pointing out to us the reference \cite{CJR14}.


\section{Linearizations and functoriality}\label{sec:generalresults}

Geometric invariant theory is an extremely powerful tool for taking quotients. Let $G$ be a reductive group acting on a scheme $X$, such that the action can be lifted to a line bundle $L$ 
on $X$. By \cite[\S 1.4]{GIT}, one can define the loci of (semi)stable points and form a GIT quotient $X /\!/_L G$ with reasonably nice properties. Although the classical results such 
as \cite[Theorem 1.10]{GIT} hold without any ampleness assumption, the theory is particularly successful for projective schemes with $G$-linearized ample line bundles. 

However, for the application that we have in mind, we will need to consider (semi)stability with respect to some \emph{semi-ample} linearizations. The classical notion of (semi)stability has its own drawback 
when it comes to functoriality. The following simple example, which is in the same spirit as \cite[Remark 4.4]{CJR14}, illustrates the failure of functoriality under the classical notion.

\begin{example}\label{ex:example1}
Let $\Gm$ act on $(\PP^1, \OO(1))$ in the standard way; namely, $\lambda \cdot [x_0:x_1] = [x_0:\lambda x_1]$. Let $\Gm$ act on $(\PP^2, \OO(1))$ trivially. Then we get an action 
of $\Gm$ on $(\PP^1 \times \PP^2, \OO(1,1))$. The stable locus is $(\PP^1 \backslash \{ 0, \infty \}) \times \PP^2$ with quotient $\PP^2$.

Let $\pi: W \to \PP^1 \times \PP^2$ be the blowup along $\PP^1 \times \{p\}$, where $p \in \PP^2$ is any closed point. Then $W$ still carries a $\Gm$-action on the first factor and $L = \pi^*\OO(1,1)$ carries 
a $\Gm$-linearization. Let $E$ be the exceptional divisor of $\pi$. We claim that no point in $E$ is semistable in the usual sense. Indeed, $E$ is a trivial $\PP^1$ bundle along the centre of the blowup. 
Let $F$ be any such $\PP^1$-fibre. Then for any $\Gm$-invariant section $s$ of a given tensor power of $L$, $s$ is constant along $F$. If $s$ does not vanish along $F$, then the 
entire $F \subseteq W_s$, where $W_s = W \backslash s^{-1}(0)$. This implies that $W_s$ cannot be affine, hence all points in $F$ are unstable.

This means that the entire exceptional divisor $E$ is excluded from the semistable locus. Indeed, the semistable locus with respect to the above linearization is isomorphic 
to $(\PP^1\backslash\{0, \infty\}) \times (\PP^2 \backslash \{p\})$ via $\pi$, so we have $W^{ss} \subsetneq \pi^{-1}((\PP^1 \times \PP^2)^{ss})$, and the GIT quotient $W /\!/_L \Gm = \PP^2 \backslash \{p\}$.
\end{example}

\begin{remark}
	In fact, Schmitt also carefully studied the failure of the functoriality of semistable loci under a very similar setup in \cite{Sch19}. He considered the projection morphism $\pi_Y: \PP^n \times Y \to Y$ 
	for some quasi-projective scheme $Y$, and observed that even under a very natural choice of ample linearizations on $Y$ and $\PP^n \times Y$, neither of $(\PP^n \times Y)^{ss}$ and $\pi_Y^{-1}(Y^{ss})$ 
	contains the other.
\end{remark}

We will see that the failure of functoriality is not ideal for our application. Indeed, we will show that, for the family of expanded degenerations considered in Section \ref{sec:application}, the (semi)stable 
locus of the relative Hilbert scheme is precisely the preimage of the (semi)stable locus of the relative symmetric product under the Hilbert-Chow morphism; see Proposition \ref{prop:ChangeSpace}. 
However, the failure of functoriality makes it harder for us to achieve such a result. In order to obtain some reasonable functorial properties, we modify the classical definition of (semi)stability as follows:

\begin{definition}
\label{def:new}
Assume that a reductive group $G$ acts on a quasi-projective scheme $X$ over an algebraically closed field $\kk$. Let $L$ be a $G$-linearized line bundle on $X$. Then a point $x \in X$ is said to be
\begin{itemize}
\item \emph{weakly semistable} if $x \in X_s$ for a $G$-invariant section $s \in \Gamma(X,L^{\otimes n})^G$, where $n$ is a positive integer and $X_s = X \backslash s^{-1}(0)$; 
\item \emph{weakly stable} if $x$ is weakly semistable, 
$G_x$ is finite, and $G \cdot x$ is closed in the semistable locus of $X$, where $G_x$ is the stabilizer of $x$, and $G \cdot x$ is the $G$-orbit containing $x$.
\end{itemize}
We denote the locus of weakly semistable and weakly stable points by $X^{wss}(L)$ and $X^{ws}(L)$ respectively.
\end{definition}

Comparing our definition with the classical definition \cite[Definition 1.7]{GIT} of (semi)stability, the only difference is that we dropped the condition which requires $X_s$ to be affine, therefore our 
definition is weaker, and allows more (semi)stable points. If $X$ is projective over an affine scheme $S$ and if $L$ is an \emph{ample} line bundle, then it is easy to see that the missing condition is 
automatically satisfied, hence the weak (semi)stability is equivalent to the classical (semi)stability in this case.
However, the difference of the two notions could be significant in general. Think of the following extreme example: let $X$ be a projective scheme with the action of a reductive group $G$; let $L = \OO_X$ 
be the trivial line bundle with the trivial $G$-linearization, namely, constant functions are $G$-invariant sections. Then no point is semistable, because $X_s$ is never affine for any invariant section $s$; but 
every point is weakly semistable, because a non-zero constant section would meet the requirement for weak semistability for every point.

In the above definition, since $X_s$ is not required to be affine any more, we are no longer able to obtain a global quotient of $X$ by gluing quotients of invariant affine charts. Therefore this notion is not 
suitable for constructing quotients (unless $L$ is ample). Instead, for our application in degenerations of Hilbert schemes, we will mainly use it for an auxiliary purpose; namely as a bridge to relate the (semi)stable 
locus on the Hilbert scheme with respect to an ample polarization to the (semi)stable locus on the symmetric product with respect to another ample polarization. The main advantage which allows us to 
do so is the strict functoriality of weak (semi)stable loci along some projective morphisms in both of the following results

\begin{lemma}
	\label{lem:quot}
	Let $W$ be a quasi-projective scheme over an algebraically closed field $\Bbbk$ of characteristic $0$. Let $H$ be a finite group acting on $W$, and $f: W \to V$ the quotient morphism. 
	Assume $G$ acts on both $W$ and $V$, such that the actions of $G$ and $H$ on $W$ commute. Assume further that $f$ is $G$-equivariant. Let $L$ be a $G$-linearized line bundle on $V$. Then we have
	\begin{align*}
		W^{wss}(f^*L) &= f^{-1}(V^{wss}(L)); \\
		W^{ws}(f^*L) &= f^{-1}(V^{ws}(L)).
	\end{align*}
\end{lemma}

\begin{proof}
	The proof is literally the same as that of \cite[Lemma 3.5]{GHHZ-2018}.
\end{proof}

\begin{lemma}
\label{lem:fun}
Assume $W$ and $V$ are schemes of finite type over a field $\Bbbk$ of characteristic $0$. Let $f: W \to V$ be a $G$-equivariant projective morphism and $L$ a $G$-linearized line bundle on $V$. 
We furthermore assume that
\begin{itemize}
	\item $V$ is irreducible and normal;
	\item the fibres of $f$ are connected;
	\item there exist $G$-equivariant open subsets $U_W \subseteq W$ and $U_V \subseteq V$, such that the complement of $U_V$ in $V$ is of codimension at least $2$, and the 
	restriction $f|_{U_W}: U_W \to U_V$ is an isomorphism.
\end{itemize}
Then we have
\begin{align*}
W^{wss}(f^*L) &= f^{-1}(V^{wss}(L)); \\
W^{ws}(f^*L) &= f^{-1}(V^{ws}(L)).
\end{align*}
\end{lemma}

\begin{proof}
We first prove the equality for semistable loci. 

The inclusion $W^{wss}(f^*L) \supseteq f^{-1}(V^{wss}(L))$ is obvious. Let $v \in V^{wss}(L)$ and $w \in f^{-1}(v)$. Then there exists a section $t \in \Gamma(V, L^{\otimes n})^G$ with $t(v) \neq 0$. 
Hence we have $f^*t \in \Gamma(W, f^*L^{\otimes n})^G$ with $(f^*t)(w) \neq 0$. 

For the inclusion $W^{wss}(f^*L) \subseteq f^{-1}(V^{wss}(L))$, let $w \in W^{wss}(f^*L)$ and $v = f(w)$. There exists a section $s \in \Gamma(W, f^*L^{\otimes n})^G$ with $s(w) \neq 0$. By identifying $U_W$ 
and $U_V$, the restriction $s|_{U_W}$ defines a section $s' \in \Gamma(U_V, L^{\otimes n})^G$. By normality we can extend $s'$ to a section $t \in \Gamma(V, L^{\otimes n})^G$, which is also $G$-invariant 
by the uniqueness of the extension. We claim that $t(v) \neq 0$. Assume on the contrary that $t(v) = 0$. It is easy to see that $f|_{\overline{U}_W}: \overline{U}_W \to V$ is surjective. 
Let $w' \in \overline{U}_W$ such that $f(w') = v$. By construction $(f^*t)|_{U_W} = s|_{U_W}$, hence $(f^*t)|_{\overline{U}_W} = s|_{\overline{U}_W}$. It follows that $s(w') = (f^*t)(w') = t(v) = 0$. 
On the other hand, since the fibres of $f$ are connected, and the restriction of $f^*L$ on each fibre is trivial, we have $s(w) = s(w') = 0$, a contradiction.

Now we prove the equality for the stable loci. For any $w \in W^{wss}(f^*L)$, we write $v = f(w)$, then by the first statement $v \in V^{wss}(L)$. We will show that $w$ is weakly stable in $W$ if and only if $v$ is 
weakly stable in $V$.

Let $w \in W^{ws}(f^*L)$. We need to show that $v \in V^{ws}(L)$. For simplicity we write $f^{wss}: W^{wss}(f^*L) \to V^{wss}(L)$ for the restriction of $f$ on the weakly semistable locus. Then the projectivity of $f$ 
implies that $f^{wss}$ is closed. Since the orbit $G \cdot w$ is closed in $W^{wss}(f^*L)$, it follows that $G \cdot v = f^{wss} (G \cdot w)$ is closed in $V^{wss}(L)$. Now we consider the commutative diagram
\begin{equation}
\label{eqn:orbdiag}
\xymatrix{
G \times \{w\} \ar[r] \ar[d]_{\cong} & G \cdot w \ar[d]^{h:=f|_{G \cdot w}} \\
G \times \{v\} \ar[r] & G \cdot v.
}
\end{equation}
By \cite[Proposition 8.3]{Hum75}, the orbits $G \cdot w$ and $G \cdot v$ are smooth, hence the morphism $h:=f|_{G \cdot w}$ is generically smooth. Furthermore $G$ acts transitively on $G \cdot w$, hence $h$ is 
smooth. In particular, $h^{-1}(v)$ is smooth. Moreover it is also projective since $f$ is a projective morphism.

We pull back the diagram \eqref{eqn:orbdiag} along the inclusion $\{v\} \into G \cdot v$ to get
\begin{equation*}
\xymatrix{
G_v \times \{w\} \ar[r]^-{\varphi} \ar[d]_{\cong} & h^{-1}(v) \ar[d]^{h|_{h^{-1}(v)} = f|_{h^{-1}(v)}} \\
G_v \times \{v\} \ar[r] & \{v\}
}
\end{equation*}
where $G_v$ denotes the stabilizer of $v$ in $G$. Notice that the action map $\varphi$ factors through the quotient group
\begin{equation*}
\xymatrix{
G_v \times \{w\} \ar[rd]_{\alpha} \ar[rr]^{\varphi} & & f^{-1}(v) \\
 & G_v/G_w \times \{w\}, \ar[ru]_{\beta} & 
}
\end{equation*}
where $G_v / G_w = \{ gG_w \mid g \in G_v \}$ is a smooth quasi-projective variety; see \cite[Theorem 4.4.1]{Bia02}. It is clear that $\beta$ is a bijection of points. Since $f^{-1}(v)$ is smooth, we conclude 
by Zariski's main theorem that $\beta$ is an isomorphism, hence $G_v/G_w$ is a projective variety. By assumption $w$ is a stable point, hence $\alpha$ is a finite group quotient which is a finite morphism. 
Thus $G_v$ is proper. However, as a closed subvariety of an affine variety $G$, $G_v$ is affine itself. It follows that $G_v$ is finite. 

For the other inclusion, assume $w \in W^{wss}(f^*L)$ and $v = f(w) \in V^{ws}(L)$. We claim that $G \cdot w$ must be closed in $W^{wss}(f^*L)$. Otherwise, there is another orbit $G \cdot w'$ in the closure 
of $G \cdot w$, which also maps to $G \cdot v$ under $f$. Moreover, $\dim (G \cdot w') < \dim (G \cdot w) \leqslant \dim G = \dim (G \cdot v)$, which is a contradiction. To show $G_w$ is finite, it suffices 
to realize that $G_w$ is a subgroup of $G_v$ which is finite itself. We conclude that $w$ is weakly stable. This finishes the proof.
\end{proof}


\section{Relative VGIT}\label{sec:relVGIT}

The main goal for this section is to prove the semi-continuity result Theorem \ref{thm:maincont}. After defining the numerical (semi)stability in \S \ref{subsec:NumCritStab}
and presenting some useful preliminary results in \S \ref{subsec:relHilbMum}, we will establish two key ingredients, which will be proven for tori in \S \ref{subsec:contM} and \S \ref{subsec:finite-ss}, and 
generalized to arbitrary reductive groups in \S \ref{subsec:reductive}. The proof of the main result will be given in \S \ref{subsec:mainproof}.

\subsection{General notation}\label{subsec:vgit-assump}

We state some assumptions and notations which are valid throughout this section.

Let $\kk$ be an algebraically closed field. Let $S$ be an affine scheme of finite type over $\kk$, and $X$ a projective $S$-scheme with structure morphism
$$ f: X \longrightarrow S. $$
Let $G$ be a reductive group acting equivariantly on $X$ and $S$. The $G$-fixed loci are denoted by $X^G$ and $S^G$ respectively. 

We denote by $\cX_\ast(G)$ the set of one-parameter subgroups of $G$. If $G = (\Gm)^n$ is a torus, it is well known that $\cX_\ast(G)$ is isomorphic to $\ZZ^n$. In this case, we write $\cX^\ast(G)$ for the 
lattice of characters of $G$, which is canonically the dual of $\cX_\ast(G)$. We also write $\cX_\ast(G)_\RR = \cX_\ast(G) \otimes_\ZZ \RR$ and $\cX^\ast(G)_\RR = \cX^\ast(G) \otimes_\ZZ \RR$. 

We denote by $\Pic^G(X)$ the group of $G$-linearized line bundles on $X$ and let $\Pic^G(X)_{\RR}$ be the vector space obtained by tensoring with $\RR$.  In analogy to  \cite[\S 2.1]{Res00} we also 
denote by $\mathrm{NS}^G(X)_\RR$ the group of $G$-linearized line bundles on $X$ modulo algebraic equivalence, tensored with $\RR$. We do not claim that in our situation $\mathrm{NS}^G(X)_\RR$ is a 
finite-dimensional vector space, nor do we need this in the subsequent arguments.

\subsection{The numerical criterion for (semi)stability}\label{subsec:NumCritStab}

We start out by recalling a generalization of the classical Hilbert-Mumford criterion for (semi)stability \cite[Theorem 2.1]{GIT}, suitable for our relative setting. This result, proved in \cite{GHH-2015}, will 
form a key tool in proving our main results. 

For an arbitrary reductive group $G$, let $L$ be a $G$-linearized line bundle on $X$. For every $x \in X$ and any $1$-PS $\lambda: \Gm \to G$, we define $\mu^L(x, \lambda)$ as follows:
\begin{itemize}
    \item if $x_0 = \lim_{t \to 0}\lambda(t) \cdot x$ exists, then $\lambda(\Gm)$ acts on the fibre $L|_{x_0}$ via a character $t \to t^{-\mu^L(x, \lambda)}$;
    \item if $\lim_{t \to 0}\lambda(t) \cdot x$ does not exist, then we formally define $\mu^L(x, \lambda) = +\infty$.
\end{itemize}

\begin{theorem} {\cite[Corollary 1.1]{GHH-2015}}
\label{thm:relativeHilbMum}
Assume that $L$ is ample. Let $ x \in X $ be a closed point. 
\begin{itemize}
\item  The point $x$ is stable if and only if $\mu^L(x, \lambda) > 0$ for any non-trivial $1$-PS $ \lambda $.
\item  The point $x$ is semistable if and only if $\mu^L(x, \lambda) \geqslant 0$ for any $1$-PS $ \lambda $.
\end{itemize}
\end{theorem}

We would like to mention here that we we made the overall assumption in \cite{GHH-2015} that the group $G$ is linearly reductive.  
This, however, is not needed for the proof of the numerical criterion where it suffices to assume that $G$ is reductive.

In \cite[Definition 4.2]{CJR14}, the authors defined the so-called \emph{numerical (semi)stability} as an alternative to the classical (semi)stability for a polarization that is not necessarily ample. 
This notion turns out to be relevant also in our relative setting, and we next give a precise definition.

\begin{definition}
\label{def:numericalsemistable}
Let $L$ be a $G$-linearized line bundle on $X$. 
\begin{itemize}
\item We say that a point $x \in X$ is \emph{numerically stable}, resp.~\emph{numerically semistable}, if condition (i), resp.~condition (ii), of Theorem \ref{thm:relativeHilbMum} holds for $x$. 

\item We denote by $ X^{ns}(L) $, resp.~$ X^{nss}(L) $, the locus of numerically stable, resp.~numerically semistable, points.
\end{itemize}
\end{definition}

Similar to our Definition \ref{def:new}, numerical (semi)stability is also equivalent to the classical notion under an ample linearization by the Hilbert-Mumford criterion, but weaker in general. 
For an ample line bundle $L$, we will therefore only talk about (semi)stable points, and use the notation $ X^{s}(L) $, resp.~$ X^{ss}(L) $.

\begin{remark}\label{rmk:nss=wss}
	The notions of weak (semi)stability in Definition \ref{def:new} and numerical (semi)stability in Definition \ref{def:numericalsemistable} coincide with each other obviously when $L$ is ample. 
	We will also see that they coincide when $L$ is the pullback of an ample line bundle along a projective morphism satisfying the assumptions of Lemma \ref{lem:fun}, which will be important for our 
	applications; see Lemma \ref{lem:nonample-HM}. However, it is not completely clear to us to what extent the two notions coincide with each other in general.
\end{remark}

\subsection{The numerical functions}\label{subsec:relHilbMum}
Next, we give a more ``intrinsic" description of the numerical function $\mu^L(x,\lambda)$, which will turn out to be quite useful for our purposes. First we need to recall some notations from \cite{GHH-2015}. 
In the remainder of this subsection we will assume that $G$ is a torus. 

We write $S = \Spec A$ for some finitely generated $\kk$-algebra $A$. Assume $L$ is an ample line bundle on $X$. By replacing $L$ with a tensor power, which does not affect stability, we can 
assume $L$ is relatively very ample. Let $V = H^0(L)$; this is a finitely generated $A$-module. Then we obtain a closed embedding
$$ X \subseteq \PP(V) = \Proj \Sym V. $$
The relative projective space $\PP(V)$ can be realized as the projectivization of the relative affine space $\VV(V) = \Spec \Sym V$.

We recall that an action by the torus $G$ on $S=\Spec A$ corresponds to a weight decomposition
\begin{equation}
	\label{eqn:weightA}
	A = \bigoplus_{\chi \in \cX^*(G)} A_\chi.
\end{equation}
A compatible $G$-action on $V$ moreover corresponds to a weight decomposition
\begin{equation}
	\label{eqn:weightV}
	V = \bigoplus_{\chi \in \cX^*(G)} V_\chi,
\end{equation}
such that the action of $A$ on $V$ respects the multi-grading.

Let $p \in \PP(V)$ be a closed point and $p^\ast \in \VV(V)$ be any of its lifts. 
The image $f(p) \in S$ is given by a $\kk$-algebra homomorphism
$$ A \to \kk, $$
which is the evaluation map of functions on $S$ at the point $f(p)$ (where we have extended $f$ in the obvious way to all of $\PP(V)$). Moreover, $p^\ast \in \VV(V)$ itself is given by an $A$-module homomorphism
$$ V \to \kk, $$
which is the evaluation map of sections of $L$ at the point $p^\ast$.
For each character $\chi \in \cX^\ast(G)$, we write the restrictions of the above homomorphisms to the corresponding component by
$$ [f(p)]_\chi: A_\chi \to \kk $$
and
$$ [p^\ast]_\chi: V_\chi \to \kk. $$

In fact, whether a given section vanishes at $p$ can be checked  on any of  its lifts $p^\ast$. Therefore, without ambiguity, in the following discussions, we can simply write the 
condition $[p^\ast]_\chi=0$ (resp. $[p^\ast]_\chi \neq 0$) by $[p]_\chi=0$ (resp. $[p]_\chi \neq 0$).

For any $1$-PS $\lambda \in \cX_\ast(G)$, we give a more detailed analysis on when $\mu^L(x,\lambda)$ achieves a finite value, and give another interpretation of this value. 

\begin{lemma}
\label{lem:fin-val}
Assume $L$ is very ample. Let $x$ be a closed point in $X$. Under the above notations, the following statements are equivalent:
\begin{itemize}
    \item [(i)] $\mu^L(x, \lambda)$ is finite;
    \item [(ii)] $\lim_{t \to 0} \lambda(t) \cdot x$ exists in $X$;
    \item [(iii)] $\lim_{t \to 0} \lambda(t) \cdot f(x)$ exists in $S$;
    \item [(iv)] $[f(x)]_\chi = 0$ for each $\chi \in \cX^\ast(G)$ satisfying $\langle \lambda, \chi \rangle <0$;
    \item [(v)] There exists some $d_0 \in \ZZ$, such that $[x]_\chi = 0$ for each $\chi \in \cX^\ast(G)$ satisfying $\langle \lambda, \chi \rangle < d_0$.
\end{itemize}
\end{lemma}

\begin{proof}
(i)$\Leftrightarrow$(ii) is by definition and (ii)$\Leftrightarrow$(iii) is due to the assumption that $f$ is projective. (iii)$\Leftrightarrow$(iv) and (iv)$\Leftrightarrow$(v) are both in \cite[Lemma 5.3]{GHH-2015}, except that 
every component $V_d$ in the single grading used in \emph{loc. cit.} should be understood as the direct sum of $V_\chi$ for all $\chi$ satisfying $\langle \lambda, \chi \rangle = d$.
\end{proof}

We will now rephrase this in a manner which will be useful later. By \cite[Lemma 1.1(a)]{Kem78}, we can find a $G$-equivariant closed embedding
$$ \iota: S \longrightarrow W $$
where $W$ is an affine space on which the $G$-action is linear. As a vector space, $W$ can be decomposed into weight spaces associated to distinct characters of $G$; namely
\begin{equation}\label{equ:introductionGamma}
 W = \bigoplus_{\chi \in \Gamma} W_\chi,
 \end{equation}
where $\Gamma$ is a non-empty finite subset of $\cX^\ast(G)$.  The set $\Gamma$ will play an important role in what follows.
Then each point $w \in W$ can be written as
$$ w = \sum_{\chi \in \Gamma} w_\chi $$
where $w_\chi \in W_\chi$ can be thought of as the coordinate of $w$ in the subspace $W_\chi$. 

For every subset $I \subseteq \Gamma$, we write
\begin{align*}
    W_I &= \{ w \in W \mid w_\chi \neq 0 \Leftrightarrow \chi \in I \}, \\
    S_I &= \iota^{-1}(W_I), \\
    X_I &= f^{-1}(S_I).
\end{align*}
Then
\begin{equation}
	\label{eqn:decompS}
	S = \bigcup_{I \subseteq \Gamma} S_I
\end{equation}
is a finite stratification indexed by the power set of $\Gamma$.

Without loss of generality, we can assume that the coordinate $w_\chi$ does not vanish identically on $S$; otherwise we can remove the summand $W_\chi$ from $W$ and embed $S$ equivariantly in a 
smaller affine space. Under this assumption, the set of characters $\chi$ with non-trivial $A_\chi$ in the decomposition \eqref{eqn:weightA} is precisely the submonoid of $\cX^\ast(G)$ generated by $\Gamma$.

Recall that in our case $V = H^0(L)$ is a finitely generated $A$-module, and let $v_1, v_2, \cdots, v_n$ be a set of generators. Without loss of generality, we can assume that each generator $v_i$ is of 
pure $G$-weight; otherwise we can replace $v_i$ by its components of pure weights, which still gives a finite set of generators for $V$. To be more precise, we write $v_i \in V_{\chi_i}$ for each $ 1 \leqslant i \leqslant n $.

Now we take an arbitrary point $x \in X$. We write
\begin{align*}
	\Gamma_{f(x)} &= \{ \chi \in \cX^\ast(G) \mid [f(x)]_\chi \neq 0 \}, \\
	\Gamma_{x} &= \{ \chi \in \cX^\ast(G) \mid [x]_\chi \neq 0 \}.
\end{align*}

We assume $x \in X_I$ for some $I \subseteq \Gamma$, or equivalently, $f(x) \in S_I$. Then we see that $\Gamma_{f(x)}$ is the submonoid of $\cX^\ast(G)$ generated by $I$. The following lemma gives a 
description of the structure of $\Gamma_x$:

\begin{lemma}
	\label{lem:strucGammax}
	Without loss of generality, we assume that 
	\begin{equation*}
		[x]_{\chi_i}(v_i) \begin{cases}
			\neq 0 & \text{if  } 1 \leqslant i \leqslant k; \\
			=0 & \text{if  } k+1 \leqslant i \leqslant n
		\end{cases}
	\end{equation*}
	for some integer $k$ with $1 \leqslant k \leqslant n$. Then we have
	$$ \Gamma_x = \bigcup_{1 \leqslant i \leqslant k} \left ( \chi_i + \Gamma_{f(x)} \right ). $$
\end{lemma}

\begin{proof}
	We first show ``$\supseteq$''. For each $1 \leqslant i \leqslant k$, the condition $[x]_{\chi_i}(v_i) \neq 0$ implies that $v_i(x) \neq 0$. For every $\chi' \in \Gamma_{f(x)}$, there exists some $a' \in A_{\chi'}$, 
	such that $a'(f(x)) \neq 0$. Then the section $a' \cdot v_i \in V$ does not vanish at $x$, which is of pure weight $\chi_i + \chi'$. It follows immediately that $\chi_i + \chi' \in \Gamma_x$.

	We next show ``$\subseteq$''. For any $\chi' \in \Gamma_x$, there exists some $v' \in V_{\chi'}$ such that $v'(x) \neq 0$. We can write
	\begin{equation}
		\label{eqn:decomVprime}
		v' = a_1 v_1 + \cdots + a_n v_n
	\end{equation}
	for some $a_1, \cdots, a_n \in A$. Without loss of generality, we can assume that $a_iv_i \in V_{\chi'}$ for every $1 \leqslant i \leqslant n$; otherwise, we can remove all components in the right-hand side 
	of \eqref{eqn:decomVprime} that are not of weight $\chi'$ and the equality still holds. Then the condition $v'(x) \neq 0$ implies that $a_{i_0}v_{i_0}$ does not vanish at $x$ for some $1 \leqslant i_0 \leqslant n$, 
	which in particular implies that $v_{i_0}(x) \neq 0$. By assumption, $v_i(x) = 0$ for each $k+1 \leqslant i \leqslant n$. Therefore we have $1 \leqslant i_0 \leqslant k$. Assume $a_{i_0} \in A$ is of weight $\chi_0$, 
	then we have $\chi' = \chi_{i_0} + \chi_0 \in \chi_{i_0} + \Gamma_{f(x)}$, as desired.
\end{proof}

For $x\in X$ we define $\mathcal{C}_x$ to be the convex cone in $\cX_\ast(G)_\RR$ generated by the set
\begin{equation}
	\label{def:Cx}
	\{ \lambda \in \cX_\ast(G) \mid \langle \lambda, \chi \rangle \geqslant 0 \text{ for all } \chi \in \Gamma_{f(x)} \}.
\end{equation}
Since $\Gamma_{f(x)}$ is a finitely generated monoid, 
$\mathcal{C}_x$ is a closed rational polyhedral cone. Moreover, by the equivalence of (i) and (iv) in Lemma \ref{lem:fin-val}, $\mu^L(x,\lambda)$ is finite for any $\lambda$ in the set \eqref{def:Cx}. Indeed, we have 
the following alternative description of $\mu^L(x, \lambda)$:

\begin{lemma}
\label{lem:geom-int}
Under the assumption of Lemma \ref{lem:fin-val}, for any $1$-PS $\lambda \in \cX_\ast(G)$, if $\mu^L(x, \lambda)$ is finite, then it is given by
\begin{equation}
	\label{eqn:afmu}
	\mu^L(x, \lambda) = - \min \{ \langle \lambda, \chi \rangle \mid \chi \in \Gamma_x \}.
\end{equation}
\end{lemma}

\begin{proof}
The claim follows immediately from \cite[Lemma 5.3]{GHH-2015}, once we notice that we can understand $V_d$ in \emph{loc. cit.} as the direct sum of $V_\chi$ for all $\chi$ satisfying $\langle \lambda, \chi \rangle = d$.
\end{proof}

One of the main benefits of Lemma \ref{lem:geom-int} is that we can use \eqref{eqn:afmu} to extend the definition of $\mu^L(x, \lambda)$ to all $\lambda \in \mathcal{C}_x$ that are not necessarily integral. 
Moreover, we see that $\mu^L(x, \lambda)$ is continuous as a function in $\lambda$. In particular, the quotient $\frac{\mu^L(x, \lambda)}{\lVert \lambda \rVert}$ is a continuous function defined on the closed 
subset of the unit sphere
$$ \mathcal{C}_x \cap \{ \lambda \in \cX_\ast(G)_\RR \mid \lVert \lambda \rVert = 1 \} $$
which is compact, hence achieves a finite maximum and a minimum if $\mathcal{C}_x \supsetneq \{0\}$.

\subsection{The numerical (semi)stability}\label{subsec:contM}
We fix a norm $\lVert \cdot \rVert$ on $\cX_\ast(G)$ which is invariant under the action of the Weyl group; see \cite[\S 1.1.3]{DH98}. Then we define
\begin{equation}
	\label{eqn:M-func}
	M^L(x) = \inf_{\lambda \in \cX_\ast(G) \backslash \{0\}} \frac{\mu^L(x,\lambda)}{\lVert \lambda \rVert}.
\end{equation}

The expression $M^L(x)$ plays an important role in VGIT for projective varieties over a field. We will next establish some of its fundamental properties also in our relative setting. Using the 
preliminaries in \S \ref{subsec:relHilbMum}, it is in fact straightforward to adapt the arguments in \cite{DH98}and \cite{Res00}. 

First we give a generalization of \cite[Proposition 1.1.6]{DH98}.

\begin{proposition}
\label{prop:M-finite}
Assume that $G$ is a torus. For any $G$-linearized line bundle $L$ on $X$, we have
\begin{itemize}
    \item if $\lim_{t \to 0} \lambda(t) \cdot x$ does not exist for any non-trivial $1$-PS $\lambda \in \cX_\ast(G)$, then $M^L(x) = + \infty$;
    \item otherwise, $M^L(x)$ has a finite value.
\end{itemize}
\end{proposition}

\begin{proof}
In the first case we have $\mu^L(x, \lambda) = + \infty$ for all $\lambda \in \cX_\ast(G)$, hence $M^L(x) = + \infty$.

In the second case we have $\mathcal{C}_x \supsetneq \{0\}$. If $L$ is ample, we have seen that $\frac{\mu^L(x, \lambda)}{\lVert \lambda \rVert}$ is a continuous function, which,
using the fact that the rational points lie dense, in particular leads to
\begin{align}
	M^L(x) &= \inf_{\lambda \in \mathcal{C}_x \backslash \{0\}} \frac{\mu^L(x,\lambda)}{\lVert \lambda \rVert} \label{eqn:newMfunc} \\
	&= \inf_{\substack{\lambda \in \mathcal{C}_x \\ \lVert \lambda \rVert = 1}} \frac{\mu^L(x,\lambda)}{\lVert \lambda \rVert} \notag
\end{align}
which achieves a finite minimum on the compact set $ \mathcal{C}_x \cap \{ \lambda \in \cX_\ast(G)_\RR \mid \lVert \lambda \rVert = 1 \} $. If $L$ is not necessarily ample, we can always find $G$-linearized 
ample line bundles $L_1$ and $L_2$, such that $L = L_1 \otimes L_2^{-1}$. Using the fact that $\mu^L(x, \lambda)$ is linear with respect to $L$, the statement follows by a similar argument as in the proof 
of \cite[Proposition 1.1.6]{DH98}.
\end{proof}

In fact, both Proposition \ref{prop:M-finite} and Proposition \ref{prop:M-cont} below hold for an arbitrary reductive group $G$; we postpone the proof to Subsection \ref{subsec:reductive}.

\begin{remark}
We emphasize that the condition for $M^L(x)$ to be finite (resp. infinite) is independent of the choice of $L$. Indeed, it depends only on whether any $1$-PS of $G$ gives a limit point when it acts on the point $x \in X$. 
\end{remark}

The function $M^{\bullet}(x): \Pic^G(X) \to \RR \cup \{+\infty \} $ 
factors through $\mathrm{NS}^G(X)$, by the argument of \cite[Proposition 2]{Res00}. The above remark then says that for any point $x \in X$, either $M^L(x)= + \infty$ for 
all $L$ or $M^{\bullet}(x): \mathrm{NS}^G(X)_{\RR} \to \RR$ is a well defined function (where we have used linearity to extend over $\RR$).

The next result is a generalization of \cite[Lemma 2]{Res00}.

\begin{proposition}
\label{prop:M-cont}
Assume that $G$ is a torus. For any point $x \in X$, we have
\begin{itemize}
    \item either $M^L(x)=+\infty$ for every $G$-linearized line bundle $L$; or
    \item the function
$$ M^{\bullet}(x): \mathrm{NS}^G(X)_\RR \longrightarrow \RR; \quad L \longmapsto M^L(x) $$
is well-defined, and its restriction to  any finite dimensional subspace of $\mathrm{NS}^G(X)_\RR$ is continuous. 
\end{itemize}
\end{proposition}

\begin{proof}
 When $M^L(x) < +\infty$
we have $\mathcal{C}_x \supsetneq \{0\}$. Notice that for every fixed $1$-PS $\lambda \in \mathcal{C}_x \backslash \{0\}$, the function $\frac{\mu^L(x,\lambda)}{\lVert \lambda \rVert}$ is a linear function in $L$. 
As the infimum of a family of linear functions, $M^L(x)$ is continuous with respect to $L$ as well.
\end{proof}

The function $M^{\bullet}(x)$ allows us to give a useful alternative description of the numerically (semi)stable locus.

\begin{lemma}
\label{lem:ample-HM}
Assume $L$ is a $G$-linearized line bundle. Then we have
\begin{align*}
    X^{nss}(L) &= \{ x \in X \mid M^L(x) \geqslant 0 \}; \\
    X^{ns}(L) &= \{ x \in X \mid M^L(x) > 0 \}.
\end{align*}
\end{lemma}
\begin{proof}
The first assertion is immediate from the definitions. Moreover, if $ M^L(x) > 0 $, then it necessarily holds that $\mu^L(x, \lambda) > 0 $ for all non-trivial $\lambda$. Thus $ x \in X^{ns}(L) $.

Assume conversely that $ x \in X^{ns}(L) $, then we want to prove that $ M^L(x) > 0 $. 
Replacing $L$ by a sufficiently large power, we write $ L = L_1 \otimes L_2^{-1} $, with each $L_i$ very ample and $G$-linearized. Suppose for contradiction that $ M^L(x) = 0 $. 
Then by Proposition \ref{prop:M-finite} and the discussion above it, there exists some $ \lambda \in \mathcal{C}_x $ such that
$$ \mu^{L_1}(x, \lambda) = \mu^{L_2}(x, \lambda). $$
We will show that we can replace $\lambda$ by a rational class $\lambda'$ such that equality still holds; this immediately contradicts $ x \in X^{ns}(L) $. Without loss of generality, we can 
assume that $G$ is a torus; otherwise we can replace $G$ by a maximal torus containing $\lambda$. 

We will need some notation and facts from \S \ref{subsec:relHilbMum}. Recall that $ L_i $ gives rise to a subset 
$\Gamma_{L_i,x}$ of $\mathcal{X}^*(G)$; we write $\Gamma_i$ for simplicity. The monoid $ \Gamma_{f(x)} $, on the other hand, is independent of the choice of line bundle. By Lemma \ref{lem:strucGammax}, there 
exist characters $ \chi_1^i, \ldots, \chi_{k_i}^i $ such that
$$ \Gamma_i = \bigcup_{1 \leq l \leq k_i} \left( \chi_l^i + \Gamma_{f(x)}  \right). $$ 

In this situation, Lemma \ref{lem:geom-int} gives that
$$ \mu^{L_i}(x, \lambda) = - \mathrm{min} \{ \langle \lambda, \chi \rangle \mid \chi \in \Gamma_i \}. $$
So $ \langle \lambda, \chi \rangle ~\geq~ - \mu^{L_i}(x, \lambda) $, where equality holds for a non-empty subset in $\Gamma_i$.

Since $ \langle \lambda, \chi' \rangle \geq 0 $ for every $ \chi' \in \Gamma_{f(x)} $, there exists some $r_i \leq k_i$ such that
$$ \langle \lambda, \chi_l^i \rangle = - \mu^{L_i}(x, \lambda) $$
for every $l = 1, \ldots, r_i $ (possibly after reordering). We can and will assume that this $r_i$ is maximal. 

Next we consider the linear subspace $\mathcal{L}$ defined by the equations 
$$ \langle \cdot , \chi_{l}^1 \rangle = \langle \cdot , \chi_{m}^2 \rangle, $$
for all choices $ l \in \{1, \ldots, r_1\} $ and $ m \in \{1, \ldots, r_2\} $. Then we can find a rational class $ \lambda' \in \mathcal{L} \cap \mathcal{C}_x $ arbitrarily close to $\lambda$. We formally 
write $ \lambda' = \lambda + \epsilon $. We claim it is still true that $ \langle \lambda', \chi_l^i \rangle = - \mu^{L_i}(x, \lambda') $ for \emph{some} $ l \in \{1, \ldots, r_i\} $. 

Indeed, if not, we could find $ \chi \in \Gamma_i $ for which the strict inequality
$$ \langle \lambda', \chi_l^i \rangle > \langle \lambda', \chi \rangle $$
holds for every $1 \leq l \leq r_i$. 
By Lemma \ref{lem:strucGammax}, we can assume $\chi \in \chi^i_{l_0} + \Gamma_{f(x)}$ for some suitable $1 \leqslant l_0 \leqslant k_i$. Then we 
have $$\langle \lambda', \chi \rangle \geqslant \langle \lambda', \chi^i_{l_0} \rangle. $$ Combining both inequalities we obtain
$$ \langle \lambda', \chi_l^i \rangle > \langle \lambda', \chi^i_{l_0} \rangle, $$
which in particular implies that $l_0 \notin \{ 1, 2, \cdots, r_i \}$. We can also rewrite the above inequality as
$$ \langle \lambda, \chi_l^i \rangle + \langle \epsilon, \chi_l^i - \chi^i_{l_0} \rangle > \langle \lambda, \chi^i_{l_0} \rangle $$
for every $1 \leqslant l \leqslant r_i$.
On the other hand, we also have
$$ - \mu^{L_i}(x, \lambda) = \langle \lambda, \chi_l^i \rangle  \leq \langle \lambda, \chi^i_{l_0} \rangle $$
for every $1 \leq l \leq r_i$.
Since there are only finitely many choices for $l_0$, there must be some choice of $l_0$ which makes the last two inequalities hold simultaneously for a sequence of $\epsilon$ arbitrarily close to $0$. 
Thus we must have $$ \langle \lambda, \chi_l^i \rangle = \langle \lambda, \chi^i_{l_0} \rangle. $$
This contradicts the maximality of the choice of $r_i$, so our claim is proved.

In conclusion, since $ \lambda' \in \mathcal{L} $, we get for suitable $l$ and $m$ that
$$ \mu^{L_1}(x, \lambda') = - \langle \lambda', \chi_l^1 \rangle = - \langle \lambda', \chi_m^2 \rangle  = \mu^{L_2}(x, \lambda'), $$
which is what we wanted to prove. 
\end{proof}

We also record the following related result, which will be helpful for our application in Section \ref{sec:application}. 

\begin{lemma}
\label{lem:nonample-HM}
Let $X$ and $Y$ be projective $S$-schemes. We assume $\pi: X \to Y$ is a $G$-equivariant projective morphism of $S$-schemes satisfying the  assumptions in Lemma \ref{lem:fun}. Let $\widetilde{L}$ be 
a $G$-linearized ample line bundle on $Y$ and $L=\pi^*\widetilde{L}$. Then we have
\begin{align*}
    X^{wss}(L) &= X^{nss}(L); \\
    X^{ws}(L) &= X^{ns}(L).
\end{align*}
\end{lemma}

\begin{proof}
Let $x \in X$ and $y=f(x) \in Y$ be closed points. 
On the one hand, by Lemma \ref{lem:fun}
we know that $x \in X^{wss}(L)$ if and only if $y \in Y^{ss}(\widetilde{L})$, which is further equivalent to $M^{\widetilde{L}}(y) \geqslant 0$ by Lemma \ref{lem:ample-HM}.
On the other hand, we claim $\mu^L(x, \lambda) = \mu^{\widetilde{L}}(y, \lambda)$ for every non-trivial $1$-PS $\lambda$. Indeed, since $X$ and $Y$ are both projective over $S$, it is clear 
that $\lim_{t \to 0} \lambda(t)\cdot x$ exists if and only if $\lim_{t \to 0} \lambda(t)\cdot y$ exists, hence $\mu^L(x, \lambda)=\infty$ if and only if $\mu^{\widetilde{L}}(y, \lambda)=\infty$. In the case when they are 
both finite, the equation follows from \cite[p.49 (iii)]{GIT}. This claim implies that $M^L(x) = M^{\widetilde{L}}(y)$ by \eqref{eqn:M-func}. In other words, $M^{\widetilde{L}}(y) \geqslant 0$ is equivalent 
to $M^L(x) \geqslant 0$. The above equivalences together conclude the first assertion. The proof for the second assertion is the same. 
\end{proof}

\subsection{Finiteness of possible semistable loci for tori}\label{subsec:finite-ss}

Throughout this subsection, we assume that $G$ is a torus. The goal is to establish our second key result in this case, namely

\begin{proposition}
\label{prop:fin-ss2}
Assume that $G$ is a torus. For any given $G$-actions on $X$ and $S$, there are only finitely many subsets of $X$ which can be realized as $X^{ss}(L)$ for some $G$-linearized ample line bundle $L$.
\end{proposition}

The analogous statement in the setting of projective varieties is a fundamental result in VGIT, proved by Dolgachev and Hu \cite[Theorem 1.3.9 (ii)]{DH98}.

The strategy for proving Proposition \ref{prop:fin-ss2} is inspired by that of \cite[Theorem 1.3.9 (ii)]{DH98}, but the situation turns out to be substantially more involved due to the action of the group $G$ on both $X$ and 
the base $S$. More precisely, we will construct a finite stratification of $X$, and show that for any $G$-linearized ample line bundle $L$, the semistable locus $X^{ss}(L)$ is always the union of a subset of the strata.

Recall that we have a stratification $\eqref{eqn:decompS}$ of $S$ indexed by the power set of $\Gamma$. For each subset $I \subseteq \Gamma$, we write $L_I$ for the sublattice of  $\cX^\ast(G)$ spanned 
by $I$, $\conv(I)$ for the convex hull of $I$ in $\cX^\ast(G)_\RR$ and $V_I = L_I \otimes_\ZZ \RR$ the linear subspace of $ \cX^\ast(G)_\RR$ spanned by $I$. Then we have
$$ \conv(I) \subseteq V_I \subseteq \cX^\ast(G)_\RR. $$
Using the duality between $\cX^\ast(G)$ and $\cX_\ast(G)$, we can realize $L_I^\perp$ as a sublattice of $\cX_\ast(G)$. Then there exists a unique subtorus $G_I$ of $G$, such that $\cX_\ast(G_I) = L_I^\perp$. 
We start with the following observation:

\begin{lemma}
	\label{lem:id_comp_torus}
	For any point $s \in S_I$, the torus $G_I$ is the identity component $G_s^\circ$ of the stabilizer group $G_s \subseteq G$.
\end{lemma}

\begin{proof}
	By \cite[Theorem in \S 16.2]{Hum75}, the identity component $G_s^\circ$ is a subtorus of $G$. We have $\cX_\ast(G_s^\circ) \subseteq L_I^\perp$ by the construction of $L_I^\perp$ 
	and $\cX_\ast(G_s^\circ) \supseteq L_I^\perp$ by the maximality of $G_s^\circ$. Hence $\cX_\ast(G_s^\circ) = L_I^\perp$, which implies $G_s^\circ = G_I$.
\end{proof}

\begin{lemma}
	\label{lem:closed-orbit}
	Let $I \subseteq \Gamma$ and $s \in S_I$ be a closed point, then the following conditions are equivalent:
	\begin{itemize}
		\item[(i)] $G \cdot s$ is a closed $G$-orbit in $S$;
		\item[(ii)] for every $\lambda \in \cX_\ast(G)$, either $\lambda$ fixes $s$ or $\lim_{t \to 0} \lambda(t) \cdot s$ does not exist;
		\item[(iii)] $\conv(I)$ as a subset of $V_I$ contains $0$ as an interior point.
	\end{itemize}
\end{lemma}

\begin{proof}
	First we show (i) $\Rightarrow$ (ii). If there exists some $\lambda \in \cX_\ast(G)$ such that $\lim_{t \to 0} \lambda(t) \cdot s = s_0 \neq s$, then $s_0 \notin G\cdot s$ since $s_0$ has a larger 
	stabilizer than $s$. Therefore $G \cdot s_0 \subseteq \overline{G \cdot s}$, which contradicts (i).
	
	Next we show (ii) $\Rightarrow$ (i). Assume on the contrary that there is a $G$-orbit $G \cdot y \subseteq (\overline{G \cdot s}) \backslash (G \cdot s)$, then by \cite[Theorem 3.6]{Bir71}, there exists 
	some $1$-PS $\lambda \in \cX_\ast(G)$, such that $\lim_{t \to 0} \lambda(t) \cdot s \in \overline{G \cdot y}$, which contradicts (ii).
	
	Now we show (ii) $\Rightarrow$ (iii). If (iii) does not hold, then there exists a (rational) hyperplane in $V_I$, such that $\conv(I)$ is contained in the closed half space on one side of this hyperplane in $V_I$, with 
	some elements of $I$ not on the hyperplane itself. Since $V_I$ is a (rational) linear subspace of $\cX^\ast(G)_\RR$, there exists a (rational) hyperplane in $\cX^\ast(G)_\RR$, such that $\conv(I)$ is contained 
	in the closed half space on one side of this hyperplane in $\cX^\ast(G)_\RR$, with some elements of $I$ not on the hyperplane itself. In other words, there exists some $\lambda \in \cX_\ast(G)$, such 
	that $\langle \lambda, \chi \rangle \geqslant 0$ for all $\chi \in I$, with strict inequalities for some $\chi \in I$. This implies $\lim_{t \to 0} \lambda(t) \cdot s$ exists and is not equal to $s$ itself, which contradicts (ii).
	
	Finally we show (iii) $\Rightarrow$ (ii). Consider an arbitrary $\lambda \in \cX_\ast(G)$. If $\lambda \in L_I^\perp$, then $\langle \lambda, \chi \rangle = 0$ for every $\chi \in I$, hence $\lambda$ fixes $s$. 
	Otherwise, the value of $\langle \lambda, \chi \rangle$ is positive for some $\chi \in I$ and negative for some other $\chi \in I$, therefore $\lim_{t \to 0} \lambda(t) \cdot s$ is divergent.
\end{proof}

\begin{definition} If the equivalent conditions (i) -- (iii) above hold for $I \subseteq \Gamma$, we shall say that $ I $ is \emph{centred}.
\end{definition}
Note that the property of being centred is independent of the choice of point $ s \in S_I $.

\begin{remark}
	\label{rmk:fix-or-infty}
Assume that $I$ is centred. We can make condition (ii) in Lemma \ref{lem:closed-orbit} more precise: for any $\lambda \in \cX_\ast(G)$, we have
	\begin{align*}
		\lambda \in L_I^\perp \ &\Longleftrightarrow \ \lambda \text{ fixes } s; \\
		\lambda \notin L_I^\perp \ &\Longleftrightarrow \ \lim_{t \to 0} \lambda(t) \cdot s \text{ does not exist}.
	\end{align*}
\end{remark}

Next, let $I \subseteq \Gamma$ be an arbitrary subset. By Lemma \ref{lem:id_comp_torus} $G_I$ fixes any closed point $s \in S_I$, thus every fibre $f^{-1}(s)$ is $G_I$-invariant. Since moreover the 
morphism $f: X \to S$ is projective, $X_I$ contains a non-empty $G_I$-fixed locus $(X_I)^{G_I}$. We have the following observation:

\begin{lemma}
	\label{lem:I-closed}
	The union of $(X_I)^{G_I}$ for all centred subsets $I \subseteq \Gamma$ is precisely the union of closed $G$-orbits in $X$.
\end{lemma}

\begin{proof}
	First of all we observe that, if $G \cdot x$ is a closed orbit for some $x \in X$, then $G \cdot f(x)$ is also a closed orbit in $S$, since the morphism $f: X \to S$ is projective. By 
	Lemma \ref{lem:closed-orbit}, $f(x) \in S_I$ where $I$ is centred, therefore $x \in X_I$ for the same $I$. 
	
	It remains to show that, for any centred $I \subseteq \Gamma$ and any point $x \in X_I$, $G \cdot x$ is a closed $G$-orbit if and only if $x \in (X_I)^{G_I}$.
	
	For one direction, we assume that $G \cdot x$ is a closed $G$-orbit. Then $(G \cdot x) \cap f^{-1}(f(x)) = G_{f(x)} \cdot x$ is also closed in the fibre $f^{-1}(f(x))$. By Lemma \ref{lem:id_comp_torus}, $G_I \cdot x$ 
	is a connected component of $G_{f(x)} \cdot x$, thus it follows that $G_I \cdot x$ is also closed in the fibre $f^{-1}(f(x))$. Since the fibre is projective, every $1$-PS of $G_I$ must fix $x$, hence $x \in (X_I)^{G_I}$.
	
	For the other direction, we assume that $x$ is fixed by $G_I$. We claim that the restriction of $f$ to the orbit $G \cdot x$ is an \'etale map onto its image $G \cdot f(x)$. Indeed, this map can be understood as the 
	natural map from $(G/G_I) \cdot x$ to $(G/G_I) \cdot f(x)$, and the stabilizer of $f(x)$ in $G/G_I$ is finite by Lemma \ref{lem:id_comp_torus}, hence the claim follows. If the orbit $G \cdot x$ were not closed, i.e. 
	there were another orbit $G \cdot y \subseteq (\overline{G \cdot x}) \backslash (G \cdot x)$, then $G \cdot f(y) \subseteq G \cdot f(x)$ since $G \cdot f(x)$ is closed by Lemma \ref{lem:closed-orbit}. However we 
	have $\dim (G \cdot f(y)) \leqslant \dim (G \cdot y) < \dim (G \cdot x) = \dim (G \cdot f(x))$. This is a contradiction because $G \cdot f(x)$ does not contain any strictly smaller $G$-orbit.
\end{proof}

Regarding closed $G$-orbits, the following lemma will be helpful later:

\begin{lemma}
	\label{lem:closure-closed}
	For any closed point $x \in X$, the closure $\overline{G \cdot x}$ of the $G$-orbit of $x$ contains some closed $G$-orbit.
\end{lemma}

\begin{proof}
	We proceed by induction on the dimension of the orbit. A $0$-dimensional $G$-orbit is a point, hence always closed. Let $x \in X$ be an arbitrary closed point. If the orbit $G \cdot x$ itself is closed then the result holds.
	 Otherwise the boundary of the orbit $(\overline{G \cdot x}) \backslash (G \cdot x)$ contains an orbit $G \cdot y$ of lower dimension. By the induction hypothesis, $\overline{G \cdot y}$ contains a closed $G$-orbit, 
	 which is also in $\overline{G \cdot x}$, as desired.
\end{proof}

\begin{remark}
Combining Lemma \ref{lem:I-closed} and Lemma \ref{lem:closure-closed}, we see that there always exists at least one centred subset of $\Gamma$.
\end{remark}

Now we are ready to construct the required stratification. Since $X_I$ is quasi-projective, the $G_I$-fixed locus $(X_I)^{G_I}$ has finitely many connected components. We write $\Lambda$ for the set of all 
connected components of $(X_I)^{G_I}$ as $I$ runs over the centred subsets of $\Gamma$. Then $\Lambda$ is a finite set of quasi-projective subschemes of $X$.

For any pair of subsets $I \subseteq \Gamma$ and $J \subseteq \Lambda$, we define
$$ X_I^J = \{ x \in X_I \mid \overline{G \cdot x} \cap Y \neq \varnothing \Leftrightarrow Y \in J, \text{ for every } Y \in \Lambda \}. $$
In other words, $X_I^J$ contains points in $X_I$ with the closure of the corresponding $G$-orbits meeting only the connected components indexed by $J$. Since both $\Gamma$ and $\Lambda$ are finite sets, we 
obtain a finite stratification
\begin{equation}
	\label{eqn:new-stratification}
	X = \bigcup_{\substack{I \subseteq \Gamma \\ J \subseteq \Lambda}} X_I^J
\end{equation}
which will be used in the proof of Proposition \ref{prop:fin-ss2}. 

\begin{remark}
	In \eqref{eqn:new-stratification} we consider each stratum $X_I^J$ simply as a subset of closed points of $X$. It is not clear to us whether it is locally closed, but this is irrelevant to the subsequent discussion.
\end{remark}

From now on the line bundle will come into play. Assume $L$ is a $G$-linearized line bundle on $X$. Let $\lambda \in \cX_\ast(G)$ be a $1$-PS of $G$ and $Y \in \Lambda$ a connected component of $(X_I)^{G_I}$ 
for some centred subset $I \subseteq \Gamma$. By Remark \ref{rmk:fix-or-infty}, there are two possibilities:
\begin{itemize}
	\item if $\lambda \notin L_I^\perp$, then $\lim_{t \to 0} \lambda(t) \cdot x$ does not exist for any $x \in X_I$;
	\item if $\lambda \in L_I^\perp$, then $\lambda$ fixes all points in $(X_I)^{G_I}$; since $Y$ is a connected component, the weight of the $\lambda$-action on the fibre $L_y$ has to be constant for all 
	points $y \in Y$, which we denoted by $w^L(Y, \lambda)$.
\end{itemize}

The following observation is essential. 

\begin{lemma}
	\label{lem:mu-value}
	For any fixed stratum $X_I^J$ in the stratification \eqref{eqn:new-stratification}, let $x \in X_I^J$ be a closed point. Then for any $G$-linearized line bundle $L$ and any $1$-PS $\lambda \in \cX_\ast(G)$, we have
	$$ \mu^L(x,\lambda) = 
		\begin{cases}
		\infty & \text{ if } \langle \lambda, \chi \rangle < 0 \text{ for some } \chi \in I; \\
		- \min \{ w^L(Y, \lambda) \mid Y \in J \} & \text{ if } \langle \lambda, \chi \rangle \geqslant 0 \text{ for all } \chi \in I.
		\end{cases} $$
\end{lemma}

\begin{proof}
	When $\langle \lambda, \chi \rangle < 0$ for some $\chi \in I$, the condition $x \in X_I$ implies that $\lim_{t \to 0} \lambda(t) \cdot x$ does not exist, hence $\mu^L(x, \lambda) = \infty$ by Lemma \ref{lem:fin-val}.
	
	When $\langle \lambda, \chi \rangle \geqslant 0$ for all $\chi \in I$, let $ x_0 = \lim_{t \to 0} \lambda(t) \cdot x. $ We compute the value of $\mu^L(x, \lambda)$ in the following two steps.
	
	\textsc{Step 1.} We show that $-\mu^L(x, \lambda) = w^L(Y, \lambda)$ for some $Y \in J$.
	
	By Lemma \ref{lem:closure-closed}, $\overline{G \cdot x_0}$ 
	contains a closed $G$-orbit, say, $G \cdot y$. Since $G \cdot y \subseteq \overline{G \cdot x}$, there exists some $Y \in J$, such that $G \cdot y \subseteq Y$ by Lemma \ref{lem:I-closed}. Notice that $x_0$ is 
	a $\lambda$-fixed point, hence every point in $\overline{G \cdot x_0}$ is a $\lambda$-fixed point. Therefore $\lambda$ acts on each fibre of $L|_{\overline{G \cdot x_0}}$ via an integral weight, which has to be 
	constant over the entire orbit closure $\overline{G \cdot x_0}$. This in particular implies that
	$$ -\mu^L(x, \lambda) = \wt_\lambda(L_{x_0}) = \wt_\lambda(L_y) = w^L(Y, \lambda), $$
	where $\wt_\lambda(-)$ is the weight of the $\lambda$-action on the corresponding line. This finishes \textsc{Step 1}.
	
	\textsc{Step 2.} We show that $-\mu^L(x, \lambda) \leqslant w^L(Y, \lambda)$ for every $Y \in J$.
	
	For this purpose, it suffices to show that $-\mu^L(x, \lambda) \leqslant -\mu^L(z, \lambda)$ for every $z \in \overline{G \cdot x}$ such that $G \cdot z$ is a closed $G$-orbit.

In order to apply Lemma \ref{lem:geom-int}, we write
\begin{align*}
    \Gamma_x &= \{ \chi \in \cX^\ast(G) \mid [x]_\chi \neq 0 \}, \\
    \Gamma_z &= \{ \chi \in \cX^\ast(G) \mid [z]_\chi \neq 0 \}.
\end{align*}
Then we get by Lemma \ref{lem:geom-int} that
\begin{align*}
    -\mu^L(x, \lambda) &= \min \{ \langle \lambda, \chi \rangle \mid \chi \in \Gamma_x \}; \\
    -\mu^L(z, \lambda) &= \min \{ \langle \lambda, \chi \rangle \mid \chi \in \Gamma_z \}.
\end{align*}
Therefore, in order to show $-\mu^L(x, \lambda) \leqslant -\mu^L(z, \lambda)$, it suffices to show that $\Gamma_z \subseteq \Gamma_x$.

For any $\chi \in \Gamma_z$, since $[z]_\chi: V_\chi \to \kk$ is non-zero, there exists some section $\sigma \in V_\chi$, such that $\sigma(z) \neq 0$. We claim that $\sigma(x) \neq 0$. Otherwise, we 
have $\sigma(x)=0$, and for any $g \in G$, $\sigma(g \cdot x) = (g^\ast \sigma)(x) = \chi(g) \cdot \sigma(x) = 0$. This means that $\sigma = 0$ on the entire orbit $G \cdot x$, hence also on its closure $\overline{G \cdot x}$. 
This is a contradiction since $\sigma(z) \neq 0$. This verifies that $\sigma(x) \neq 0$, hence $[x]_\chi: V_\chi \to \kk$ is non-zero, so $\chi \in \Gamma_x$. This finishes \textsc{Step 2}.
	
The above two steps conclude the second case in the lemma.
\end{proof}

We are now ready to prove the main result of this subsection.

\begin{proof}[Proof of Proposition \ref{prop:fin-ss2}]
	By Lemma \ref{lem:mu-value}, we see that for any $G$-linearized line bundle $L$ on $X$, the function on the lattice of $1$-PS's
	$$ \mu^L(x, -): \cX_\ast(G) \lra \ZZ \cup \{\infty\} $$
	is the same function for all points $x \in X_I^J$. 
	
	Assume $L$ is an ample line bundle. By Theorem \ref{thm:relativeHilbMum}, (semi)stability of points in $X_I^J$ are all the same, hence $X^{ss}(L)$ must be a union of a subset of 
	strata in \eqref{eqn:new-stratification}, as desired.
\end{proof}

Finiteness moreover holds also for the stable and unstable loci, we leave the easy verification of this fact to the reader.

\begin{corollary}
\label{cor:fin-all2}
The finiteness statement in Proposition \ref{prop:fin-ss2} also holds for stable loci and unstable loci with respect to $G$-linearized ample line bundles.
\end{corollary}

\subsection{Generalization to arbitrary reductive groups}\label{subsec:reductive}
The purpose of this subsection is to demonstrate that the main results in \S\ref{subsec:contM} and \S\ref{subsec:finite-ss} remain true without assuming that $G$ is a torus. We follow mostly the arguments 
in \cite[\S 1.1]{DH98} and \cite[\S 1.2]{Res00}.

Let $G$ be an arbitrary reductive group, and $T$ a fixed maximal torus of $G$. Then every maximal torus of $G$ can be given by $g^{-1}Tg$ for some $g \in G$; see e.g. \cite[Theorem 17.87]{Mil17}. 

\begin{proposition}
	\label{prop:268reductive}
	Proposition \ref{prop:M-finite} and Proposition \ref{prop:M-cont} hold for an arbitrary reductive group $G$.
\end{proposition}

\begin{proof}
	In this proof we need to emphasize the dependence of the function $M^L(x)$ on the group $G$. So the function defined in \eqref{eqn:M-func} will be denoted by $M^L_G(x)$. On the other hand, if we only 
	consider the nontrivial $1$-PS's of the maximal torus $T$, then the corresponding function will be denoted by $M^L_T(x)$.
	
We only need to show that $M^L_G(x)>-\infty$. Indeed, by a similar argument as in \cite[Proposition 1.1.6]{DH98} we have 
$$ M^L_G(x) = \inf_{g \in G} M^L_T(gx). $$ 
By Proposition \ref{prop:M-finite}, $M^L_T(y)> -\infty$ for every point $y \in X$. Moreover, by Proposition \ref{prop:fin-ss2}, for any fixed $G$-linearized line bundle $L$ (which in particular is $T$-linearized), the 
function $\mu^L(y, -): \cX_\ast(T) \to \ZZ \cup \{\infty\}$ stays the same when $y$ runs over all closed points of any fixed stratum $X_I^J$. It follows that $M^L_T(y)$ is a constant function on each stratum $X_I^J$, 
hence takes only finitely many different values in $\RR \cup \{\infty\}$ when $y$ runs over all closed points of $X$, which implies that $M^L_G(x) = \inf_{g \in G} M^L_T(gx) > -\infty$.
\end{proof}

\begin{proposition}
	\label{prop:21622reductive}
	Proposition \ref{prop:fin-ss2} and Corollary \ref{cor:fin-all2} hold for an arbitrary reductive group $G$.
\end{proposition}
\begin{proof}
The proof follows \cite[Remark 1.3.10]{DH98}, we omit the details.
\end{proof}

\subsection{Semi-continuity}\label{subsec:mainproof}

The semi-continuity of (semi)stable loci is a powerful result in VGIT; see e.g. \cite[Proposition 4]{Res00} as well as \cite[Lemma 3.10]{Laza-2013} and the references therein. We generalize this property to the 
relative setting, including the situations where the relevant line bundle on the central fibre is not necessarily ample. In fact, we will give three formulations to allow some flexibility in applications.

\begin{theorem}
	\label{thm:weakcont}
	Let $L_t$ be a $G$-linearized line bundle on $X$ for every $t$ satisfying $0 \leqslant t \leqslant 1$. We assume that the function $M^{L_t}(x)$ is continuous in $t$ for every point $x \in X$, and that $L_t$ is 
	ample for $0 < t \leqslant 1$.
	Then we have
	$$ X^{ns}(L_0) \subseteq X^s(L_t) \subseteq X^{ss}(L_t) \subseteq X^{nss}(L_0) $$
	 for $0 < t \ll 1$.
\end{theorem}
\begin{proof}
	We follow the proof of \cite[Proposition 4]{Res00}. The middle inclusion is obvious. We prove the other two. 
	
	\textsc{Step 1.} We
	first prove $X^{ns}(L_0) \subseteq X^s(L_t)$. 

	By Proposition \ref{prop:21622reductive}, there are finitely many possible subsets of $X$, say, $X^s_1, X^s_2, \cdots, X^s_n$, which could be realized as the stable loci for some $G$-linearized ample line bundle 
	on $X$. Namely, for each $0 < t \leqslant 1$, $X^s(L_t)$ must be one of them.

	We assume that $X^{ns}(L_0) \not\subseteq X^s_i$ for each $1 \leqslant i \leqslant p$ and $X^{ns}(L_0) \subseteq X^s_j$ for each $p+1 \leqslant j \leqslant n$. Then for each $1 \leqslant i \leqslant p$, we can find 
	some point $x_i \in X^{ns}(L_0) \backslash X^s_i$.

	We have earlier proved that $x_i \in X^{ns}(L_0)$ implies $M^{L_0}(x_i) >0$. Since $M^{L_t}(x_i)$ is continuous in $t$, there exists some $\varepsilon_i>0$, such that $M^{L_t}(x_i) >0$ for 
	every $0 < t < \varepsilon_i$. By Lemma \ref{lem:ample-HM}, we have $x_i \in X^s(L_t)$ for every $0 < t < \varepsilon_i$. Since $x_i \not\in X^s_i$, we conclude that $X^s(L_t) \neq X^s_i$ for 
	every $0 < t < \varepsilon_i$.

	Let $\varepsilon = \min \{ \varepsilon_1, \cdots, \varepsilon_p \}$. Then for every $0 < t < \varepsilon$, $X^s(L_t)$ must be one of $X^s_{p+1}, \cdots X^s_n$. In particular, we have $X^{ns}(L_0) \subseteq X^s(L_t)$ 
	for every $0 < t < \varepsilon$.

	\textsc{Step 2.} We now prove $X^{ss}(L_t) \subseteq X^{nss}(L_0)$, which can be reformulated in terms of unstable loci as $X^{nus}(L_0) \subseteq X^{us}(L_t)$.
Indeed, the proof in this step is exactly the same as in \textsc{Step 1}, if we replace every occurrence of ``(numerically) stable'' in the proof by ``(numerically) unstable'', and replace every occurrence of $M^L(x_i)>0$ in 
the proof by $M^L(x_i)<0$. This finishes the proof of the statement.
\end{proof}

The same proof applies to the following discrete version of Theorem \ref{thm:weakcont}:

\begin{theorem}
	\label{thm:disccont}
	Let $L_m$ be a $G$-linearized ample line bundle on $X$ for each positive integer $m$, and $L_\infty$ a $G$-linearized line bundle on $X$. 
	We assume
	$$ \lim_{m \to \infty} M^{L_m}(x) = M^{L_\infty}(x) $$
	for every point $x \in X$. Then we have
	$$ X^{ns}(L_\infty) \subseteq X^s(L_m) \subseteq X^{ss}(L_m) \subseteq X^{nss}(L_\infty) $$
	for $m \gg 0$. \qed
\end{theorem}

The following version is a more familiar reformulation of Theorem \ref{thm:weakcont}:

\begin{theorem}
\label{thm:maincont}
Let $L_0$ and $L_1$ be $G$-linearized line bundles on $X$, and
$$ L_t = L_0^{1-t} \otimes L_1^t $$
for any $0 \leqslant t \leqslant 1$.
Assume further that $L_t$ is ample for $0 < t \leqslant 1$.
Then we have
$$ X^{ns}(L_0) \subseteq X^s(L_t) \subseteq X^{ss}(L_t) \subseteq X^{nss}(L_0) $$
for $0 < t \ll 1$.
\end{theorem}

\begin{proof}
	We note that the line bundles $L_t$ are all contained in a $2$-dimensional subspace of $\mathrm{NS}^G(X)_\RR$. By Proposition \ref{prop:268reductive} the function $M^{L_t}(x)$ is continuous in $t$  for 
	every point $x \in X$.
	The claim then follows from Theorem \ref{thm:weakcont}.
\end{proof}


\section{Application to degeneration of Hilbert schemes} \label{sec:application}

In this section, we apply the general theory of the previous sections to study degenerations of Hilbert schemes in the framework developed in \cite{GHH-2016,GHHZ-2018}.

\subsection{Setup}\label{subsec:appsetup}

We first fix some notations which will be used throughout this section; more details can be found in \cite{GHHZ-2018}. 

Let $\kk$ be an algebraically closed field of characteristic $0$, $C$ a smooth affine curve over $\kk$ and $f \colon X \to C$ a projective strict simple degeneration. We do not make any assumption on the 
relative dimension of $f$. By applying Jun Li's technique of expanded degenerations, we obtain a new family
\begin{equation}
	\label{eqn:ExpandedFamily}
	f[n]: X[n] \longrightarrow C[n],
\end{equation}
where $X[n]$ and $C[n]$ are both smooth, $C[n]$ is of dimension $n+1$, and the fibres of $f[n]$ are ``expansions'' of the fibres of $f$. Moreover, the group $G= (\Gm)^n$ acts equivariantly on $X[n]$ and $C[n]$. Let
$$ g \colon \Hilb^n(X[n]/C[n]) \longrightarrow C[n] $$
be the relative Hilbert scheme and
$$ \cZ \subseteq X[n] \times_{C[n]} \Hilb^n(X[n]/C[n]) $$
the universal closed subscheme with natural projections to the two factors denoted by $q$ and $p$ respectively. Then we obtain the commutative diagram
\begin{equation*}
	\begin{tikzcd}
		\mathcal{Z} \ar[r,"q"] \ar[d,"p"'] & X[n] \ar[d,"{f[n]}"] \\
		\Hilb^n(X[n]/C[n]) \ar[r,"g"'] & C[n].
	\end{tikzcd}
\end{equation*}

Let $\cL$ be the $G$-linearized ample line bundle on $X[n]$ constructed in \cite{GHH-2016}, then for every positive integer $\ell$, we can define
$$ \cL_m = \det p_\ast q^\ast \cL^m, $$
which is an ample line bundle on $\Hilb^n(X[n]/C[n])$ for $\ell \gg 0$ by \cite[\S 2.2.1]{GHH-2016}. On the other hand, the line bundle $\cL^{\boxtimes n}$ on the $n$-fold 
product $X[n] \times_{C[n]} \cdots \times_{C[n]} X[n]$ descends to an ample line bundle $\widetilde{\cL}$ on the relative symmetric product $\Sym^n(X[n]/C[n])$ by \cite[Lemma 3.1]{GHHZ-2018}. Via the relative 
Hilbert-Chow morphism (see e.g. \cite[Paper III, \S 4.3]{Rydh})
$$ \pi \colon \Hilb^n(X[n]/C[n]) \longrightarrow \Sym^n(X[n]/C[n]), $$
we obtain a semi-ample line bundle
$$ \cL_\infty = \pi^\ast \widetilde{\cL} $$
on $\Hilb^n(X[n]/C[n])$. Notice that $\cL_\infty$ is not ample, because it is not positive on any curve that is contracted by $\pi$. The relative Hilbert scheme $\Hilb^n(X[n]/C[n])$, along with the line bundles $\cL_m$ 
and $\cL_\infty$, carries a naturally induced $G$-action. 

Recall that in \cite[Theorem 2.9]{GHH-2016}, the (semi)stable locus of the relative Hilbert scheme $\Hilb^n(X[n]/C[n])$ with respect to the $G$-linearized line bundle $\cL_m$ for sufficiently large $\ell$ was computed. 
One striking phenomenon is that, whether a point $[Z] \in \Hilb^n(X[n]/C[n])$ is (semi)stable only depends on its underlying cycle. In fact, this is not a coincidence. The goal of this section is to determine 
the (semi)stable locus of $\Hilb^n(X[n]/C[n])$ from an alternative perspective, which gives a conceptual interpretation for the irrelevance of the scheme structure of $[Z]$ to (semi)stability.

More precisely, our starting point is the relation
\begin{equation}
	\label{eqn:ssforSym}
	\Sym^n(X[n]/C[n])^{ss}(\widetilde{\cL}) = \Sym^n(X[n]/C[n])^s(\widetilde{\cL})
\end{equation}
which was proved in \cite[Proposition 3.3]{GHHZ-2018}. We will relate this to the (semi)stable locus of $\Hilb^n(X[n]/C[n])$ with respect to $\cL_m$ for $m \gg 0$ in two steps using the general theory developed 
in previous sections.

\subsection{From symmetric product to Hilbert scheme}\label{subsec:ExpStep1}

We first determine the weakly (semi)stable loci of $\HilbXC$ with respect to $\cL_\infty$.

\begin{proposition}
	\label{prop:ChangeSpace}
	We have
	\begin{align*}
		&\quad\ \HilbXC^{wss}(\cL_\infty) = \HilbXC^{ws}(\cL_\infty) \\
		&= \pi^{-1}(\SymXC^{ss}(\widetilde{\cL})) = \pi^{-1}(\SymXC^{s}(\widetilde{\cL})).
	\end{align*}
\end{proposition}

\begin{proof}
	The claims follows from \eqref{eqn:ssforSym} and Lemma \ref{lem:fun}. It remains to check that all three assumptions in Lemma \ref{lem:fun} are satisfied by $\pi$.
	
	First of all, since the morphism $f[n]$ in \eqref{eqn:ExpandedFamily} is flat, the structure morphism $\phi$ from the $n$-fold product \eqref{eqn:nfoldProd} is also flat. Moreover, since the generic f
	ibre of $\phi$ is irreducible, it follows that the $n$-fold product $X[n] \times_{C[n]} \cdots \times_{C[n]} X[n]$ is irreducible by \cite[Proposition 4.3.8]{liu-2002}. Therefore as a finite group quotient, $\Sym^n(X[n]/C[n])$ 
	is also irreducible. The normality of $\Sym^n(X[n]/C[n])$ will be proven in Lemma \ref{lem:SymNormal}.
	
	The connectivity of fibres of $\pi$ follows from \cite[Proposition 2.3]{fogarty-1968}.
	
	Finally, let $U_H$ and $U_S$ be the open subset in $\Hilb^n(X[n]/C[n])$ and $\Sym^n(X[n]/C[n])$ respectively, parametrising $n$-tuples of pairwise distinct points in the smooth locus of the morphism $f[n]$ 
	in \eqref{eqn:ExpandedFamily}. Since the smooth locus of the morphism $f[n]$ is $G[n]$-invariant, both $U_H$ and $U_S$ are also $G[n]$-invariant. It is easy to see that the complement of $U_S$ 
	in $\Sym^n(X[n]/C[n])$ is of codimension $2$, and the restriction $\pi|_{U_H}: U_H \to U_S$ is an isomorphism. Hence the desired statement follows.
\end{proof}

The following lemma was required in the above proof; it might also be of some independent interest.

\begin{lemma}
	\label{lem:SymNormal}
	The scheme $\Sym^n(X[n]/C[n])$ is normal.
\end{lemma}

\begin{proof}
	Without loss of generality, we can assume that $C[n] = \AA^{n+1}$.  We first claim that 
	the singular locus of the $n$-fold product $X[n] \times_{C[n]} \cdots \times_{C[n]} X[n]$ has codimension $2$.
	For this we consider the natural morphism
	\begin{equation}
		\label{eqn:nfoldProd}
		\varphi: X[n] \times_{C[n]} \cdots \times_{C[n]} X[n] \longrightarrow C[n].
	\end{equation}
	Note that any point in $X[n] \times_{C[n]} \cdots \times_{C[n]} X[n]$ representing a tuple of $n$  points in the smooth locus of the morphism $f[n]: X[n] \to C[n]$ is a smooth point 
	in $X[n] \times_{C[n]} \cdots \times_{C[n]} X[n]$. 
	
	For any closed point $q \in C[n] = \AA^{n+1}$, the fibre $(f[n])^{-1}(q)$ is smooth if no coordinate of $q$ vanishes and singular in codimension $1$ if $q$ has at least one vanishing coordinate. Hence the intersection 
	of the singular locus of $X[n] \times_{C[n]} \cdots \times_{C[n]} X[n]$ with any closed fibre $\varphi^{-1}(q)$ is empty if $q$ has no vanishing coordinate; and of codimension $1$ otherwise. This shows the claim.

	Next we claim that $X[n] \times_{C[n]} \cdots \times_{C[n]} X[n]$ is a local complete intersection in a smooth variety.
	In order to see this we use the 
	closed embedding 
	$$\iota: X[n] \times_{C[n]} \cdots \times_{C[n]} X[n] \longrightarrow X[n] \times \cdots \times X[n].$$
	Since $X[n]$ is smooth, the target variety of $\iota$ is a smooth variety. It thus suffices to show that the number of equations required for the closed embedding $\iota$ agrees with the codimension of the 
	closed embedding $\iota$.
	
	Assume the relative dimension of $f[n]: X[n] \to C[n]$ is $k$, then we have $\dim X[n] = n+k$, hence we obtain
	\begin{align*}
		\dim X[n] \times_{C[n]} \cdots \times_{C[n]} X[n] &= n+kn, \\
		\dim X[n] \times \cdots \times X[n] &= n(n+k).
	\end{align*} 
	It follows that the codimension of the closed embedding $\iota$ is $n^2-n$.
	
	Now we take any closed point $(p_1, \cdots, p_n) \in X[n] \times_{C[n]} \cdots \times_{C[n]} X[n]$. Since $X[n]$ is smooth of dimension $n+k$, in a neighbourhood of each point $p_i \in X[n]$, we pick a local 
	chart of coordinates $\{x_{i,1}, \cdots, x_{i,n+k}\}$ on $X[n]$. The morphism $f[n]: X[n] \to C[n]$ is then given by the equations
	$$ t_j = f_{ij}(x_{i,1}, \cdots, x_{i,n+k}) $$
	for $j = 1, 2, \cdots, n$, where $t_1, \cdots, t_n$ are the coordinates of $C[n] \cong \AA^{n+1}$. It follows that locally near the point $(p_1, \cdots, p_n)$, the closed embedding $\iota$ is defined by the equations
	$$ f_{1j}(x_{1,1}, \cdots, x_{1,n+k}) = \cdots = f_{nj}(x_{n,1}, \cdots, x_{n,n+k}) $$
	for $j = 1, 2, \cdots, n$, which is a total of $n(n-1)$ equations, as claimed.
	
It now follows from Serre's $R_1+S_2$ criterion, see \cite[Proposition II.8.23]{hartshorne-1977} that $X[n] \times_{C[n]} \cdots \times_{C[n]} X[n]$ is normal. But then $\Sym^n(X[n]/C[n])$ is also normal as this 
property is preserved under finite group actions.
	\end{proof}

\subsection{From semi-ample  to ample line bundles}\label{subsec:ExpStep2}

We then determine the (semi)stable loci of $\HilbXC$ with respect to $\cL_m$ for $m \gg 0$.

\begin{proposition}
	\label{prop:ChangeBundle}
	For $m \gg 0$ we have 
	\begin{align*}
		&\quad\ \HilbXC^{wss}(\cL_\infty) = \HilbXC^{ws}(\cL_\infty) \\
		&= \HilbXC^{ss}(\cL_m) = \HilbXC^{s}(\cL_m).
	\end{align*}
\end{proposition}

\begin{proof}
	We first observe that
	\begin{align*}
		&\quad\ \HilbXC^{ns}(\cL_\infty) = \HilbXC^{ws}(\cL_\infty) \\
		&= \HilbXC^{wss}(\cL_\infty) = \HilbXC^{nss}(\cL_\infty),
	\end{align*}
	where the middle equality is due to Proposition \ref{prop:ChangeSpace}, and the other two are due to Lemma \ref{lem:nonample-HM}.
	Moreover, $\cL_m$ is ample for $m \gg 0$ by \cite[\S 2.2.1]{GHH-2016}. The following Lemma \ref{lem:newcMfun} guarantees that all assumptions in Theorem \ref{thm:disccont} are satisfied. Therefore we conclude
	\begin{align*}
		&\quad\ \HilbXC^{ns}(\cL_\infty) \subseteq \HilbXC^s(\cL_m) \\
		&\subseteq \HilbXC^{ss}(\cL_m) \subseteq \HilbXC^{nss}(\cL_\infty).
	\end{align*}
	Since the first and the last sets are equal, it follows that all of them must be equal.
\end{proof}

The following results concerning the continuity of weights and $M^{\bullet}(-)$ were required in the above proof.

\begin{lemma}
	\label{lem:newcmufun}
	For any point $z \in \Hilb^n(X[n]/C[n])$, and for any $1$-PS $\lambda \in \cX_\ast(G)$, we have
	$$ \lim_{m \to \infty} \mu^{\frac{\cL_m}{m}}(z, \lambda) = \mu^{\cL_\infty}(z, \lambda). $$
\end{lemma}

\begin{proof}
	There are two possible cases. If $\lim_{t \to 0} \lambda(t) \cdot z$ does not exist, then we have
	$$ \mu^{\frac{\cL_m}{m}}(z, \lambda) = \infty = \mu^{\cL_\infty}(z, \lambda) $$
	for every positive integer $m$. The statement follows automatically.
	
	Otherwise, assume that
	\begin{equation}
		\label{eqn:above1}
		\lim_{t \to 0} \lambda(t) \cdot z = z_0 \in \Hilb^n(X[n]/C[n])
	\end{equation}
	and assume further that $z_0$ is represented by a closed subscheme $Z_0 \subseteq X[n]$ of length $n$. We can decompose the corresponding cycle as a sum of positive multiples of distinct points
	\begin{equation}
		\label{eqn:above2}
		[Z_0] = \sum_p n_p[p].
	\end{equation}
	
	By \cite[\S 2.2.2]{GHH-2016}, the fibre of $\cL_m$ at the point $z_0$ is given by
	\begin{align*}
		\cL_m(z_0) &= \wedge^n H^0(\OO_{Z_0} \otimes \cL^{\otimes m}) \\
		&= \wedge^n H^0(\OO_{Z_0}) \otimes \left( \otimes_p \cL(p)^{\otimes n_p} \right)^{\otimes m}.
	\end{align*}
	By definition we have
	\begin{align*}
		\mu^{\cL_m}(z, \lambda) &= - \wt_\lambda(\cL_m(z_0)) \\
		&= - \wt_\lambda(\wedge^n H^0(\OO_{Z_0})) - m \cdot \sum_p n_p \cdot \wt_\lambda(\cL(p)),
	\end{align*}
	where $\wt_\lambda (-)$ represents the weight of the $\lambda$-action on the corresponding line. This immediately gives
	\begin{equation}
		\label{eqn:newweightLm}
		\mu^{\frac{\cL_m}{m}}(z, \lambda) = - \frac{1}{m} \cdot \wt_\lambda(\wedge^n H^0(\OO_{Z_0})) - \sum_p n_p \cdot \wt_\lambda(\cL(p)).
	\end{equation}
	
	On the other hand, from the construction of the line bundle $\cL_\infty$, c.f. \cite[Remark 3.2]{GHHZ-2018}, it is immediately clear that the fibre of $\cL_\infty$ at the point $z_0$ is given by
	$$ \cL_\infty(z_0) = \otimes_p \cL(p)^{\otimes n_p}. $$
	Therefore we have
	\begin{equation}
		\label{eqn:newweightLinfty}
		\mu^{\cL_\infty}(z, \lambda) = - \sum_p n_p \cdot \wt_\lambda (\cL(p)).
	\end{equation}
	
	Comparing \eqref{eqn:newweightLm} and \eqref{eqn:newweightLinfty}, we have
	$$ \lim_{m \to \infty} \mu^{\frac{\cL_m}{m}}(z, \lambda) = \mu^{\cL_\infty}(z, \lambda), $$
	as desired.
\end{proof}

\begin{lemma}
	\label{lem:newcMfun}
	For any point $z \in \Hilb^n(X[n]/C[n])$, we have
	$$ \lim_{m \to \infty} M^{\frac{\cL_m}{m}}(z) = M^{\cL_\infty}(z). $$
\end{lemma}

\begin{proof}
	Without loss of generality, we assume there exists some $\lambda \in \cX_\ast(G) \backslash \{0\}$, such that $\lim_{t \to 0} \lambda(t) \cdot z$ exists; otherwise both sides are $+\infty$. 
	For any such $\lambda$, we use the notations in \eqref{eqn:above1} and \eqref{eqn:above2} to define
	$$ f_\lambda(t) = \frac{1}{\lVert \lambda \rVert} \left(- t \cdot \wt_\lambda(\wedge^n H^0(\OO_{Z_0})) - \sum_p n_p \cdot \wt_\lambda(\cL(p)) \right). $$
	We further define
	$$ f(t) = \inf_{\lambda \in \cX_\ast(G) \backslash \{0\} } f_\lambda(t). $$
	By Proposition \ref{prop:M-finite}, the function $f(t)$ achieves finite values. Moreover, since $f_\lambda(t)$ is a linear function in $t$ for each $\lambda$, as the infimum of a collection of linear functions, $f(t)$ is also 
	a continuous function in $t$. 
	
	Moreover, by \eqref{eqn:newweightLm} and \eqref{eqn:newweightLinfty} we have
	$$ M^{\frac{\cL_m}{m}}(z) = \inf_{\lambda \in \cX_\ast(G) \backslash \{0\} } \frac{\mu^{\frac{\cL_m}{m}}(z, \lambda)}{\lVert \lambda \rVert} = \inf_{\lambda \in \cX_\ast(G) \backslash \{0\} } f_\lambda\left(\frac{1}{m}\right) = f\left(\frac{1}{m}\right) $$
	and
	$$ M^{\cL_\infty}(z) = \inf_{\lambda \in \cX_\ast(G) \backslash \{0\} } \frac{\mu^{\cL_\infty}(z, \lambda)}{\lVert \lambda \rVert} = \inf_{\lambda \in \cX_\ast(G) \backslash \{0\} } f_\lambda(0) = f(0). $$
	The continuity of $f(t)$ immediately implies that
	$$ \lim_{m \to \infty} M^{\frac{\cL_m}{m}}(z) = M^{\cL_\infty}(z) $$
	as desired.
\end{proof}

\begin{remark}
Proposition \ref{prop:ChangeBundle} would also follow from Theorem \ref{thm:maincont} if we could prove
\begin{equation}
	\label{eqn:newlimlb}
	\lim_{m \to \infty} \frac{\cL_m}{m} = \cL_\infty
\end{equation}
in $\mathrm{NS}^G(\Hilb^n(X[n]/C[n]))_\QQ$. Although we strongly believe that \eqref{eqn:newlimlb} should hold, a rigorous proof is not known to us. To achieve Proposition \ref{prop:ChangeBundle}, we instead 
applied Theorem \ref{thm:disccont}. In the course of proof we used Lemmas \ref{lem:newcmufun} and \ref{lem:newcMfun}, which can be seen as a weaker version of \eqref{eqn:newlimlb}.
\end{remark}

To summarize, we obtain an alternative and more conceptual proof of \cite[Theorem 2.9]{GHH-2016}, namely

\begin{corollary}\label{cor:pullbackagree}
For $m \gg 0$ we have
	\begin{align*}
		&\quad\ \HilbXC^{ss}(\cL_m) = \pi^{-1}(\SymXC^{ss}(\widetilde{\cL}))  \\
		&= \HilbXC^{s}(\cL_m) = \pi^{-1}(\SymXC^{s}(\widetilde{\cL})).
	\end{align*}
\end{corollary}

\begin{proof}
	This is a combination of Propositions \ref{prop:ChangeSpace} and \ref{prop:ChangeBundle}.
\end{proof}

\bibliography{hilbdeg}{}
\bibliographystyle{alpha}

\end{document}